\documentclass[10pt,a4paper]{amsart} 
\usepackage{cite}
\usepackage{amsaddr}
\usepackage{amssymb,amsthm,amsmath,amsfonts,amscd}
\usepackage{hyperref} 
\usepackage{url}
\usepackage[english]{babel}
\usepackage{enumerate}
\usepackage{esint}
\usepackage{bbm}
\usepackage{mathrsfs}

\newcommand{\norm}[1]{\left\lVert#1\right\rVert}

\newcommand{\abs}[1]{\lvert#1\rvert}
\newcommand{\lrabs}[1]{\left\lvert#1\right\rvert}

\newcommand{\lrbr}[2]{\left\langle#1,{#2}\right\rangle}

\newcommand{\gdualdel}[4]{\sideset{_{#1}}{_{#2}}{\mathop{\left\langle{#3},{#4}\right\rangle}}}

\newcommand{\up}[1]{\textup{#1}}

\renewcommand{\tilde}{\widetilde}
\renewcommand{\hat}{\widehat}

\renewcommand{\d}{d}
\newcommand{\D}{\textup{D}}

\newcommand{\eps}{{\varepsilon}}
\renewcommand{\phi}{\varphi}

\newcommand{\R}{\mathbbm{R}}
\newcommand{\N}{\mathbbm{N}}

\newcommand{\sign}{\operatorname{sign}}
\newcommand{\esssup}{\operatorname{ess\;sup}}
\newcommand{\essinf}{\operatorname{ess\;inf}}
\renewcommand{\div}{\operatorname{div}}

\newcommand{\supp}{\operatorname{supp}}

\newcommand{\diam}{\operatorname{diam}}

\newcommand{\osc}{\operatorname{osc}}

\newcommand{\loc}{\textup{loc}}

\newcommand{\BIGOP}[1]{\mathop{\mathchoice%
{\raise-0.22em\hbox{\huge $#1$}}%
{\raise-0.05em\hbox{\Large $#1$}}{\hbox{\large $#1$}}{#1}}}

\newcommand{\BIGboxplus}{\mathop{\mathchoice%
{\raise-0.35em\hbox{\huge $\boxplus$}}%
{\raise-0.15em\hbox{\Large $\boxplus$}}{\hbox{\large $\boxplus$}}{\boxplus}}}

\numberwithin{equation}{section}

\newtheorem{theorem}{Theorem}[section] 

\theoremstyle{plain}
\newtheorem{defi}[theorem]{Definition}
\newtheorem{lem}[theorem]{Lemma}

\newtheorem{prop}[theorem]{Proposition}
\newtheorem{rem}[theorem]{Remark}

\newtheorem{thm}[theorem]{Theorem}

\newtheorem{cor}[theorem]{Corollary}
\newtheorem*{bem*}{Bemerkung}

\newcommand{\ol}[1]{{#1}^\mu}

\begin{document}
\title[Uniqueness of weighted Sobolev spaces]{Uniqueness of weighted Sobolev spaces with weakly differentiable weights}
\author[J.M. T\"olle]{Jonas M. T\"olle}
\address{Institut f\"ur Mathematik, Technische Universit\"at Berlin (MA 7-5)\\
Stra\ss{}e des 17. Juni 136, 10623 Berlin, Germany}
\email{jonasmtoelle@gmail.com}
\thanks{The research was partly supported by the German Science Foundation (DFG), IRTG 1132, ``Stochastics and Real World Models'' and the Collaborative Research Center 701 (SFB 701), ``Spectral Structures and Topological Methods in Mathematics'', Bielefeld.}
\begin{abstract}
We prove that weakly differentiable weights $w$ which, together with their reciprocals, satisfy certain local integrability conditions, admit a
unique associated first-order $p$-Sobolev space, that is
\[H^{1,p}(\R^d,w\,\d x)=V^{1,p}(\R^d,w\,\d x)=W^{1,p}(\R^d,w\,\d x),\]
where $d\in\N$ and $p\in [1,\infty)$.
If $w$ admits a (weak) logarithmic gradient $\nabla w/w$ which is in $L^q_\loc(w\,\d x;\R^d)$, $q=p/(p-1)$, we propose an alternative definition of the weighted $p$-Sobolev
space based on an integration by parts formula involving $\nabla w/w$.
We prove that weights of the form $\exp(-\beta\abs{\cdot}^q-W-V)$ are $p$-admissible, in particular, satisfy a
Poincar\'e inequality, where $\beta\in (0,\infty)$, $W$, $V$ are convex and bounded below such that $\abs{\nabla W}$ satisfies a growth condition (depending on $\beta$ and $q$) and $V$ is bounded.
We apply the uniqueness result to weights of this type. The associated nonlinear degenerate evolution equation is also discussed.
\end{abstract}

\maketitle

\section{Introduction}

Consider the following quasi-linear PDE in $\R^d$ (in the weak sense)
\begin{equation}\label{pde}-\div\left[w\abs{\nabla u}^{p-2}\nabla u\right]=fw,\end{equation}
(here $1<p<\infty$) where $w\ge 0$ is a locally integrable function, the \emph{weight} and $f$ is sufficiently regular
(e.g $f\in L^q(w\,\d x)$, see below).
Let $\mu(\d x):=w\,\d x$, $q:=p/(p-1)$.
The nonlinear weighted $p$-Laplace operator involved in \eqref{pde} can be identified with the G\^ateaux derivative of the
convex functional
\begin{equation}\label{energy}E_0^\mu:u\mapsto\frac{1}{p}\int\abs{\nabla u}^p\,\d\mu.\end{equation}
By methods well known
in calculus of variations, solutions to \eqref{pde} are characterized by minimizers of the convex functional
\begin{equation}
E_f^\mu:u\mapsto E_0^\mu(u)-\int fu\,\d\mu.
\end{equation}
Of course, the minimizer obtained depends on the energy space chosen for the
functional \eqref{energy}. It is natural to demand that the space of test functions $C_0^\infty$ is included in this energy space (where we take the subscript zero to denote functions with \emph{compact support} rather than functions \emph{vanishing at infinity}).

Therefore, let $H^{1,p}(\mu)$ be the completion of $C_0^\infty$ w.r.t. the Sobolev norm
\[\norm{\cdot}_{1,p,\mu}:=\left(\norm{\nabla\cdot}_{L^p(\mu;\R^d)}^p+\norm{\cdot}_{L^p(\mu)}^p\right)^{1/p}.\]
$H^{1,p}(\mu)$ is referred to as the so-called \emph{strong weighted Sobolev space}.
Of course, in order to guarantee that $H^{1,p}(\mu)$ will be a space of functions
we need a ``closability condition'', see equation \eqref{clos} below.

Let $V$ be a weighted Sobolev space such that
\begin{itemize}
 \item $V\subset L^p(\mu)$ densely and continuously,
 \item $V$ admits a linear gradient-operator $\nabla^V:V\to L^p(\mu;\R^d)$ that respects $\mu$-classes,
 \item $V$ is complete w.r.t. the Sobolev norm,
 \item $C_0^\infty\subset V$ and $\nabla u=\nabla^V u$ $\mu$-a.e. for $u\in C_0^\infty$ and hence $H^{1,p}_0(\mu)\subset V$.
\end{itemize}
In the case that
\[H^{1,p}(\mu)\subsetneqq V,\]
the so-called \emph{Lavrent'ev phenomenon}, first described in \cite{Lavr},
occurs if
\[\min_{u\in V}E_f(u)<\min_{u\in H^{1,p}(\mu)}E_f(u).\]
This leads to different variational solutions to equation \eqref{pde}, as discussed in detail in \cite{Past}.
In order to prevent this possibility, we are concerned with the problem
\[H^{1,p}(\mu)=V,\]
which is equivalent to the density of $C_0^\infty$ in $V$ and therefore is called
``smooth approximation''. Classically, if $w\equiv 1$, the solution to this problem is
known as the Meyers-Serrin theorem \cite{MS} and briefly denoted by $H=W$.
If $p=2$, the problem is also known as ``Markov uniqueness'', see \cite{ARZ2,ARZ,Eb,RZ3,RZ4}.

$H=W$ for weighted Sobolev spaces ($p\not=2$) has been studied e.g. in \cite{PC,Kil,Zhi3}.
$H=W$ is in particular useful for identifying a Mosco limit \cite{K1,Toe3}

We are going to investigate two types of weighted Sobolev spaces substituting $V$.

Let $\phi:=w^{1/p}$. Consider following condition for $p\in [1,\infty)$
\[\tag{\textbf{Diff}}\phi\in W^{1,p}_\loc(\d x),\quad\beta:=p\frac{\nabla\phi}{\phi}\in L^q_\loc(\mu;\R^d)\]
Assuming (\textbf{Diff}), we shall define the Sobolev space $V^{1,p}(\mu)$ (which
extends $H^{1,p}(\mu)$) by saying
that $f\in V^{1,p}(\mu)$ if
$f\in L^p(\mu)$ and
there is a gradient
\[\ol{\nabla}f:=(\ol{\partial}_1 f,\ldots,\ol{\partial}_d f)\in L^p(\mu;\R^d)\]
such that the integration by parts formula
\begin{equation}\label{ibpintintr}
\int \ol{\partial}_i f\eta\,\d\mu=-\int f\partial_i\eta\,\d\mu-\int f\eta\beta_i\,\d\mu
\end{equation}
holds for all $\eta\in C_0^\infty(\R^d)$ and all $i\in\{1,\ldots,d\}$.
We point out that, in general, we do not expect $f\in L^{1}_{\loc}(dx)$! Therefore we cannot use distributional
derivatives here. Formula \eqref{ibpintintr} is based on the weak derivative of $fw$ rather
than on that of $f$, see Section \ref{ibpsubsec}
for details.

For $p=2$, this framework has been carried out by Albeverio et al.
in \cite{AKR,AR,AR3,ARZ}.

Assuming (\textbf{Diff}),
equation \eqref{pde} has the following heuristic reformulation
\[-\div\left[\abs{\nabla u}^{p-2}\nabla u\right]-\lrbr{\abs{\nabla u}^{p-2}\nabla u}{\beta}=f,\]
which suggests that \eqref{pde} can be regarded as a first-order perturbation of
the unweighted $p$-Laplace equation. In these terms, \eqref{pde} mimics a nonlinear
Kolmogorov operator.

Let us state our main result.
\begin{thm}\label{CFthm}
Assume \up{(\textbf{Diff})}. If $p=1$, assume additionally that
\begin{equation}\label{LipGrad}\nabla \phi\in L^\infty_\loc(\d x;\R^d).\end{equation}
Then $C_0^\infty(\R^d)$ is dense in $V^{1,p}(\mu)$, and, in particular,
\[H^{1,p}(\mu)=V^{1,p}(\mu).\]
\end{thm}
For $p=2$, Theorem \ref{CFthm} was proved by R\"ockner and Zhang \cite{RZ3,RZ4} using methods from
the theory of Dirichlet forms depending strongly on the $L^2$-framework.
For weights of the type $\mu(\d x)= Z^{-1}e^{-U(x)}\,\d x$, $Z:=\int e^{-U(x)}\,\d x$, Lorenzi
and Bertoldi proved Theorem \ref{CFthm} under much stronger differentiability assumptions,
see \cite[Theorem 8.1.26]{LoreBert07}. We also refer to Chapter 2.6 of Bogachev's book \cite{Bogachev10} for related results.

Our proof is carried out in Section \ref{CFsect} and
inspired by the work of Patrick Cattiaux and
Myriam Fradon \cite{CF}. In contrary to their proof, in which Fourier transforms are used (relying on the $L^2$-framework),
we shall use maximal functions in order to obtain the fundamental uniform estimate.
Of course,
formula \eqref{ibpintintr} provides highly useful for the proof.

Consider the following well known condition for $p\in [1,\infty)$
\[\tag{\textbf{Reg}}\left.\begin{aligned}\phi^{-q}\in L^1_\loc(\d x),&\quad\text{for $p\in (1,\infty)$}\\
\phi^{-1}\in L^\infty_\loc(\d x),&\quad\text{for $p=1$}\end{aligned}\right\}.\]

Condition (\textbf{Reg}) (``regular'') implies that each Sobolev function is a \emph{regular}
\emph{(Schwartz)} \emph{distribution}, see Section \ref{Kufsec}.

Let $\D$ be the gradient in the sense of Schwartz distributions. Assuming (\textbf{Reg}), we define
\[W^{1,p}(\mu):=\left\{u\in L^p(\mu)\;\vert\;\D u\in L^p(\mu;\R^d)\right\},\]
see e.g. \cite{KO}. We shall refer to $W^{1,p}(\mu)$ as the so-called \emph{Kufner-Sobolev space}, due to [26], and remark that its definition is the standard one in the literature of weighted Sobolev spaces.
It is well known that $H^{1,p}(\mu)=W^{1,p}(\mu)$ is implied by the famous \emph{$p$-Muckenhoupt condition}, due to \cite{Muck72}, in symbols $w\in A_p$, $1<p<\infty$, where $A_p$ is defined as follows:
$w=\phi^p\in A_p$ if and only if there is a global constant $K>0$ such that
\begin{equation}\label{Muckeq}
\left(\fint_B \phi^p\,\d x\right)\cdot\left(\fint_B \phi^{-q}\,\d x\right)^{p-1}\le K,\end{equation}
for all balls $B\subset\R^d$. See Proposition \ref{Muckprop} below for the proof.
We refer to the lecture notes by Bengt Ove Turesson \cite{Tur} for a detailed discussion of the
class $A_p$. See also \cite[Ch. 15]{HKM}.

As a consequence of Theorem \ref{CFthm}, we obtain the following result:
\begin{cor}\label{maincor}
Assume \up{(\textbf{Reg})}, \up{(\textbf{Diff})}, and if $p=1$ assume also that \eqref{LipGrad} holds. Then
\[H^{1,p}(\mu)=V^{1,p}(\mu)=W^{1,p}(\mu).\]
\end{cor}
We shall give a precise proof in Section \ref{Kufsec}.
As an application, we investigate the evolution problem related to PDE \eqref{pde} in
Section \ref{evo_sec}. In particular, we provide existence and uniqueness of the following (global)
evolution equation in $L^2(\mu)$, $p\ge 2$,
\begin{equation}\label{eeq}\left.\begin{aligned}\partial_t u&=\frac{1}{w}\div\left[w\abs{\nabla u}^{p-2}\nabla u\right],&&\quad\text{in}\;\;(0,T)\times\R^d,\\
u(\cdot,0)&=u_0\in L^2(\mu),&&\quad\text{in}\;\;\R^d.\end{aligned}\right\}\end{equation}

See \cite{AMRT} for an example of the (local and nonlocal) weighted evolution problem with Muckenhoupt weights. We also refer the work by Hauer and Rhandi \cite{HauRha}, who prove a non-existence result for the global weighted evolution problem.

An application related to nonlinear potential theory and the elliptic equation \eqref{pde}
is given by
the notion of $p$-admissibility, as introduced by Heinonen, Kilpel\"ainen and Martio in \cite{HKM}
(see Definition \ref{paddefi} below).

We say that a function $F:\R^d\to\R$ \emph{has property} (D), if there are constants $c_1\ge 1$, $c_2\in\R$ such that
$F(2x)\le c_1 F(x)+c_2$. If $F$ is concave, it has property (D) with $c_1=2$ and $c_2=F(0)$. With the help of the ideas of
Hebisch and Zegarli\'nski \cite{HeZe} we are able to prove:

\begin{thm}\label{p-ad-thm}
Let $1<p<\infty$, $q:=p/(p-1)$. Let $\beta\in (0,\infty)$, let $W\in C^1(\R^d)$ be bounded below
and suppose that
\[\abs{\nabla W(x)}\le\delta\abs{x}^{q-1}+\gamma\]
for some $\delta<\beta q$ and $\gamma\in (0,\infty)$. Suppose also that $-W$ has property \up{(D)}.
Let $V:\R^d\to\R$ be a measurable function such that $\osc V:=\sup V-\inf V<\infty$ and $-V$ has property \up{(D)}.

Then
\[x\mapsto\exp(-\beta\abs{x}^q-W(x)-V(x))\]
is a $p$-admissible weight. If, additionally, $V\in W^{1,\infty}_\loc(\d x)$, this weight satisfies the
conditions of Corollary \ref{maincor}.
\end{thm}

\begin{rem}
If $V$ is convex, then $V$ is locally Lipschitz by \cite[Theorem 10.4]{Rock4} and hence $V\in W^{1,\infty}_\loc(\d x)$ by
\cite[\S 4.2.3, Theorem 5]{EG}.
\end{rem}

\begin{rem}
If $\osc V<\infty$, then the weight $\exp(-V)$ obviously satisfies Muckenhoupt's condition \eqref{Muckeq}
for all $1<p<\infty$.
\end{rem}

As an application of the main result \ref{CFthm}, the weighted Poincar\'e inequality
\[\int\lrabs{f-\frac{\int f\,w\,\d x}{\int\,w\,\d x}}^p\,w\,\d x\le c\int\abs{\nabla f}^p\,w\,\d x,\]
for the weight
$w:=\exp(-\beta\abs{\cdot}^q-W-V)$ also holds for $f\in V^{1,p}(w\,\d x)$ and for $f\in W^{1,p}(w\,\d x)$. We also point out, that by Kinderlehrer and Stampacchia \cite{KinSta} the stationary problem \eqref{pde} can be solved for $p$-admissible
weights, see \cite[Ch. 17, Appendix I]{HKM}.

\subsection*{Notation}

Equip $\R^d$ with the Euclidean norm $\abs{\cdot}$ and the Euclidean scalar product $\lrbr{\cdot}{\cdot}$.
For $i\in\{1,\ldots,d\}$, denote by $e_i$ the $i$-th unit vector in $\R^d$. For $\R^d$-valued functions $v$
we indicate the projection on the $i$-th coordinate by $v_i$.
We denote the (weak or strong)
partial derivative $\frac{\partial}{\partial x_i}$ by $\partial_i$. Also $\nabla:=(\partial_1,\ldots,\partial_d)$.
Denote by $C^\infty=C^\infty(\R^d)$, $C_0^\infty=C_0^\infty(\R^d)$ resp., the spaces of infinitely often continuously differentiable functions on $\R^d$, with compact support resp.
We denote the standard Sobolev spaces (local Sobolev spaces resp.) on $\R^d$ by $W^{1,p}(\d x)$ and $W^{1,p}_\loc(\d x)$,
with $1\le p\le\infty$.

For $x\in\R^d$, let
\[\sign(x):=\begin{cases}\dfrac{x}{\abs{x}},&\;\;\text{if}\;\;x\not=0,\\
                          0,&\;\;\text{if}\;\;x=0.
                         \end{cases}\]
Denote by $\D$ the gradient in the sense of Schwartz distributions. For $x\in\R^d$ and $\rho>0$, set
$B(x,\rho):=\big\{y\in\R^d\;\big\vert\;\abs{x-y}<\rho\big\}$ and
$\overline{B}(x,\rho):=\big\{y\in\R^d\;\big\vert\;\abs{x-y}\le\rho\big\}$.
With a \emph{standard mollifier} we mean a family of functions $\{\eta_\eps\}_{\eps>0}$ such that
\[\eta_\eps(x):=\frac{1}{\eps^d}\eta\left(\frac{x}{\eps}\right),\]
where $\eta\in C_0^\infty(\R^d)$ with $\eta\ge 0$,
$\eta(x)=\eta(\abs{x})$, $\supp\eta\subset\overline{B}(0,1)$ and $\int\eta\,\d x=1$.

\section{Weighted Sobolev spaces}

For all what follows, fix $1\le p<\infty$ and $d\in\{1,2,\ldots\}$. Set $q:=p/(p-1)$.

\begin{defi}
For an a.e.-nonnegative measurable function $f$ on $\R^d$, we define the \emph{regular set}
\[R(f):=\left\{y\in\R^d\;\bigg\vert\;\int_{B(y,\eps)}\frac{1}{f(x)}\,\d x<\infty\;\;\text{for some}\;\eps>0\right\},\]
where we adopt the convention that $1/0:=+\infty$ and $1/+\infty:=0$.

Define also
\[\hat{R}(f):=\left\{y\in\R^d\;\bigg\vert\;\esssup\displaylimits_{x\in B(y,\eps)}\frac{1}{f(x)}<\infty\;\;\text{for some}\;\eps>0\right\}.\]
\end{defi}

Obviously, $R(f)$ is the largest open set $O\subset\R^d$, such that $1/f\in L^1_\loc(O)$.
Also, it always holds that $f>0$ $\d x$-a.e. on $R(f)$. $\hat{R}(f)$ is the largest open set $\hat{O}\subset\R^d$
such that $1/f\in L^\infty_{\loc}(\hat{O})$. By abuse of notation, we denote the regular set for functions $\psi:\mathbb{R}\to\mathbb{R}$ by the same symbol.

Fix a \emph{weight} $w$, that is a measurable function $w\in L^1_\loc(\R^d)$, $w\ge 0$ a.e. Set $\mu(\d x):=w\,\d x$.
Following the notation of \cite{RZ3}, we set $\phi:=w^{1/p}$.

\begin{defi}\label{hamdefi}
Consider the following conditions:
\begin{enumerate}
\item[\up{(\textbf{Ham1})}] For each $i\in\{1,\ldots,d\}$ and for ($(d-1)$-dimensional) Lebesgue a.a. $y\in\{e_i\}^\perp$ it holds that the map
$\psi_y: t\mapsto\phi(y+te_i)$ satisfies $\psi_y^p(t)=0$ for $\d t$-a.e. $t\in\R\setminus R(\psi_y^q)$ if $p\in (1,\infty)$ and satisfies $\psi_y(t)=0$ for $\d t$-a.e. $t\in\R\setminus \hat{R}(\psi_y)$ if $p=1$.
\item[\up{(\textbf{Ham2})}] $\phi^p(x)=0$ for $\d x$-a.e. $x\in\R^d\setminus R(\phi^q)$ if $p\in (1,\infty)$ and $\phi(x)=0$ for
$\d x$-a.e. $x\in \R^d\setminus \hat{R}(\phi)$ if $p=1$.

\end{enumerate}
\end{defi}

Both (\textbf{Ham1}), (\textbf{Ham2})
are called \emph{Hamza's condition} (``on rays'' resp. ``on $\R^d$''), due to \cite{Ham}.

It is straightforward that the following implications hold
\[(\textbf{Reg})\quad\Longrightarrow\quad(\textbf{Ham2})\quad\Longrightarrow\quad(\textbf{Ham1}).\]
Also, if (\textbf{Reg}) holds, $\mu$ and $\d x$ are equivalent measures.

\begin{rem}\label{essinfrem}
Suppose that for $\d x$-a.a. $x\in\{\phi^p>0\}$,
\[\essinf\displaylimits_{y\in B(x,\delta)}\phi^p(y)>0\]
for some $\delta=\delta(x)>0$.
Then \up{(\textbf{Ham2})} holds (and is indeed equivalent to \up{(\textbf{Ham2})} for $p=1$). In particular, \up{(\textbf{Ham2})} holds whenever
$\phi^p\ge 0$ is lower semi-continuous.
\end{rem}

The following lemma is analogous to \cite[Lemma 2.1]{AR}.
\begin{lem}\label{locallyintegrablelem}
Assume that \up{(\textbf{Ham2})} holds. Then for $p\in (1,\infty)$,
\[L^p(\R^d,\mu)\subset L^1_\loc(R(\phi^q),\d x)\]
continuously and for $p=1$
\[L^1(\R^d,\mu)\subset L^1_\loc(\hat{R}(\phi),\d x),\]
continuously.
\end{lem}
\begin{proof}
Let $u\in L^p(\R^d,\mu)$ and let $B\Subset R(\phi^q)$ be a ball. By H\"older's inequality, if $p\in (1,\infty)$,
\[\int_B\abs{u}\,\d x\le\left(\int_{R(\phi^q)}\abs{u}^p\,\phi^p\,\d x\right)^{1/p}\cdot\left(\int_B \phi^{-q}\,\d x\right)^{1/q}.\]
$\int_B \phi^{-q}\,\d x$ is finite by (\textbf{Ham2}).
For $p=1$, just observe that for balls $B\Subset\hat{R}(\phi)$
\[\int_B\abs{u}\,\d x\le\left(\int_{\hat{R}(\phi)}\abs{u}\,\phi\,\d x\right)\cdot\left(\esssup\displaylimits_{x\in B}\frac{1}{\phi(x)}\right).\]

\end{proof}

\begin{defi}
Let
\[X:=\left\{u\in C^\infty(\R^d)\;\Big\vert\;\norm{u}_{1,p,\mu}:=\left(\norm{\nabla u}_{L^p(\mu;\R^d)}^p+\norm{u}_{L^p(\mu)}^p\right)^{1/p}<\infty\right\}.\]
Let $H^{1,p}(\mu):=\tilde{X}$ be the abstract completion of $X$ w.r.t. the pre-norm $\norm{\cdot}_{1,p,\mu}$.
\end{defi}

\begin{lem}\label{closlem}
Suppose that \up{(\textbf{Ham1})} holds. Then for all sequences $\{u_n\}\subset C^\infty$ the following condition holds:
\begin{equation}\label{clos}
\begin{split}
&\lim_n\norm{u_n}_{L^p(\mu)}=0\;\text{and}\;\{u_n\}\;\text{is}\;\norm{\nabla\cdot}_{L^p(\mu;\R^d)}\text{-Cauchy}\\
&\qquad\text{always imply}\\
&\lim_n\norm{\nabla u_n}_{L^p(\mu;\R^d)}=0.
\end{split}
\end{equation}
Condition \eqref{clos} is referred to as \emph{closability}.
\end{lem}
\begin{proof}
We shall consider partial derivatives first. Fix $i\in\{1,\ldots,d\}$.

Let $\{u_n\}\in C^\infty$ such that $\norm{u_n}_{L^p(\mu)}\to 0$ and such that $\{u_n\}$ is $\norm{\partial_i\cdot}_{L^p(\mu)}$-Cauchy. By the Riesz-Fischer theorem, $\{\partial_i u_n\}$
converges to some $v\in L^p(\mu)$. Fix $y\in\{e_i\}^\perp$. Set $\psi_y:t\mapsto \phi(y+te_i)$. By (\textbf{Ham1}) and
Lemma \ref{locallyintegrablelem} for $d=1$, setting $I_y:=R(\psi^q_y)$, if $p\in (1,\infty)$ and $I_y:=\hat{R}(\psi_y)$ if $p=1$,
we conclude that the sequence of maps $\{t\mapsto\partial_i u_n(y+te_i)\}$ converges to $t\mapsto v(y+te_i)$ in $L^1_\loc(I_y)$.
Let $\eta\in C_0^\infty(I_y)$,
\begin{align*}0&=\lim_n\int_{I_y} u_n(y+te_i)\frac{\d}{\d s}\eta(s)\Big\vert_{s=t}\,\d t=-\lim_n\int_{\supp \eta\cap I_y}(\partial_i u_n)(y+te_i) \eta(t)\,\d t\\
&=-\int_{\supp \eta\cap I_y} v(y+t e_i) \eta(t)\,\d t.\end{align*}
We conclude that $v(y+t e_i)=0$ for $\d y$-a.e. $y\in\{e_i\}^\perp$ and $\d t$-a.e $t\in I_y$. By (\textbf{Ham1}) it follows that
$v=0$ $\mu$-a.e. on $\R^d$.

Assume now that $\{u_n\}\in C^\infty$ such that $\norm{u_n}_{L^p(\mu)}\to 0$ and such that $\{u_n\}$ is $\norm{\nabla\cdot}_{L^p(\mu;\R^d)}$-Cauchy.
Clearly each $\{\partial_i u_n\}$ is a Cauchy-sequence in $L^p(\mu)$.
Therefore, for some constant $C=C(p,d)>0$,
\[\int_{\R^d}\abs{\nabla u_n}^p\,\d\mu\le C\sum_{i=1}^d\int_{\R^d}\abs{\partial_i u_n}^p\,\d\mu\to 0,\]
as $n\to\infty$ by the arguments above.
\end{proof}

\begin{prop}
Assume \up{(\textbf{Ham1})}. Then $H^{1,p}(\mu)$ is a space of $\mu$-classes of functions and
is continuously embedded into $L^p(\mu)$. Also, $H^{1,p}(\mu)$ is separable and reflexive whenever $p\in (1,\infty)$.
\end{prop}
\begin{proof}
The proof works by similar arguments as in the unweighted case.
\end{proof}

Denote the (class of the) gradient of an element $u\in H^{1,p}(\mu)$ by $\nabla^\mu u$.

\begin{prop}\label{c0prop}
Assume \up{(\textbf{Ham1})}. The $\mu$-classes of $C_0^\infty(\R^d)$ functions are dense in $H^{1,p}(\mu)$.
\end{prop}
\begin{proof}
The proof is a standard localization argument using partition of unity, see e.g. \cite[Theorem 1.27]{HKM}.
\end{proof}

\subsection{Integration by parts}\label{ibpsubsec}

We follow the approach of Albeverio, Kusuoka and R\"ockner \cite{AKR}, which is to define a weighted
Sobolev space via an integration by parts formula. Recall that $w=\phi^p$. A function $f\in L^p(\mu)$
might fail to be a Schwartz distribution. Instead, consider $f\phi^p=(f\phi)\phi^{p-1}$, which is in $L^1_\loc(\d x)$
by H\"older's inequality and therefore $\D(f\phi^p)$ is well defined.
For $f\in C_0^\infty$, the Leibniz formula yields
\begin{equation}\label{Leibeq}(\nabla f)\phi^p=\D(f\phi^p)-pf\frac{\D\phi}{\phi}\phi^p,\end{equation}
which motivates the definition of the \emph{logarithmic derivative} of $\mu$:
\[\beta:=p\frac{\D\phi}{\phi},\]
where we set $\beta\equiv 0$ on $\{\phi=0\}$.
The name arises from the (solely formal) identity $\beta=\nabla(\log(\phi^p))$.

\begin{lem}\label{pqlem}
Condition \up{(\textbf{Diff})} implies $\phi^p\in W^{1,1}_\loc(\d x)$ and
\begin{equation}\label{bbeq}\beta=p\frac{\nabla\phi}{\phi}=\frac{\nabla(\phi^p)}{\phi^p},\end{equation}
where $\nabla$ denotes the usual weak gradient.

Moreover, $\beta\in L^p_\loc(\mu;\R^d)$ and, if $p\in(1,\infty)$, $\abs{\nabla\phi}\phi^{p-2}\in L^q_\loc$.
\end{lem}
\begin{proof}
For $p=1$, the claim follows from \up{(\textbf{Diff})}.
Assume \up{(\textbf{Diff})} and that $p\in (1,\infty)$. $\phi^p\in L^1_\loc$ is clear.
We claim that
\begin{equation}\label{ppeq}\nabla(\phi^p)=p\phi^{p-1}\nabla\phi.\end{equation}
Let $\phi_\eps:=\eta_\eps\ast\phi$, where $\{\eta_\eps\}$ is a standard mollifier.
It follows from the classical chain rule that for all $\eps>0$
\[\nabla((\phi_\eps)^p)=p\phi_\eps^{p-1}\nabla\phi_\eps.\]
Since $\phi^{p-1}\in L^q_\loc$ and $\nabla\phi\in L^p_\loc$, we can pass to the limit
in $L^1_\loc$ and get that $\phi^p\in W^{1,1}_\loc(\d x)$.
\eqref{ppeq} follows now from the uniqueness of the gradient in $W^{1,1}_\loc(\d x)$.
The first equality in \eqref{bbeq} is clear. The second follows from \eqref{ppeq}. $\beta\in L^p_\loc(\mu;\R^d)$
is clear.
The last equality follows from (\textbf{Diff}) by
\[\lrabs{\frac{\nabla\phi}{\phi}}^{q}\phi^p=\left(\abs{\nabla\phi}\phi^{p-2}\right)^q.\]
\end{proof}

\begin{lem}\label{Wqlem}
Assume \up{(\textbf{Diff})} and that $p\in (1,\infty)$. Then $\phi^{p-1}\in W^{1,q}_\loc(\d x)$. Also,
\[\nabla(\phi^{p-1})=(p-1)\phi^{p-2}\nabla\phi.\]
\end{lem}
\begin{proof}
Fix $1\le i\le d$. For $N\in\N$, define $\psi_N:\R\to\R$ by $\psi_N(t):=(\abs{t}\vee N^{-1}\wedge N)^{p-1}$.
Clearly, $\psi_N$ is a Lipschitz function.
By the chain rule for Sobolev functions \cite[Theorem 2.1.11]{Ziem},
\[\partial_i\psi_N(\phi)=(p-1)1_{\{N^{-1}\le\phi\le N\}}\frac{\phi^{p-1}}{\phi}\partial_i\phi.\]
We have that $\psi_N(\phi)\to\phi^{p-1}$ $\d x$-a.s. as $N\to\infty$. Also,
\[\abs{\psi_N(\phi)}^q\le\abs{(\phi\vee N^{-1})^p}\le C\abs{\phi}^p+C\in L^1_\loc.\]
Furthermore, by Lemma \ref{pqlem},
\[\lrabs{1_{\{N^{-1}\le\phi\le N\}}\frac{\phi^{p-1}}{\phi}\partial_i\phi}\le\abs{\phi^{p-2}\partial_i\phi}\in L^q_\loc.\]
Hence by Lebesgue's dominated convergence theorem, $\psi_N(\phi)\to\phi^{p-1}$ in $L^q_\loc$
and $\partial_i\psi_N(\phi)\to(p-1)\phi^{p-2}\partial_i\phi$ in $L^q_\loc$.
The claim is proved.
\end{proof}

\begin{lem}\label{wabscontlem}
Fix $1\le i\le d$.
Suppose that \up{(\textbf{Diff})} holds. Then there is a version $\widetilde{\phi^p}$ of $\phi^p$, such that for $y\in\{e_i\}^\perp$ the
map $\widetilde{\psi^p_y}:t\mapsto\widetilde{\phi^p}(y+te_i)$
is absolutely continuous for almost all $y\in\{e_i\}^\perp$. Furthermore, for almost all $y\in\{e_i\}^\perp$, setting $\psi_y:t\mapsto \phi(y+te_i)$,
\[\R\setminus R(\psi^q_y)\supset\{t\in\R\,\vert\,\widetilde{\psi^p_y}(t)=0\},\]
if $p\in (1,\infty)$ and
\[\R\setminus \hat{R}(\psi_y)\supset\{t\in\R\,\vert\,\widetilde{\psi^1_y}(t)=0\}\]
if $p=1$.
Recall that in both cases the $\d t$-almost sure inclusion ``$\,\subset\,$'' holds automatically.
\end{lem}
\begin{proof}
Note that $\phi^p\in W^{1,1}_\loc(\d x)$ by Lemma \ref{pqlem}.
Then the first part follows from a well known theorem due to Nikod\'ym, cf. \cite[Theorem 2.7]{Miz}.
The second part follows from absolute continuity and Remark \ref{essinfrem} for $d=1$.
\end{proof}

We immediately get that:

\begin{cor}\label{diffcor}
It holds that
\[\up{(\textbf{Diff})}\quad\Longrightarrow\quad\up{(\textbf{Ham1})}.\]
\end{cor}

Motivated by \eqref{Leibeq}, we shall define the weighted Sobolev space $V^{1,p}(\mu)$.

\begin{defi}
If \up{(\textbf{Diff})} holds, we define the space $V^{1,p}(\mu)$ to be the set of all $\mu$-classes of functions $f\in L^p(\mu)$
such that there exists a gradient
\[\nabla^\mu f=(\partial_1^\mu f,\ldots,\partial_d^\mu f)\in L^p(\mu;\R^d)\]
which satisfies
\begin{equation}\label{ibpeq}
\int \partial_i^\mu f\eta\phi^p\,\d x=-\int f\partial_i\eta\phi^p\,\d x-\int f\eta\beta_i\phi^p\,\d x
\end{equation}
for all $i\in\{1,\ldots,d\}$ and all $\eta\in C_0^\infty(\R^d)$.

Define also $V^{1,p}_\loc(\mu)$ by replacing
$L^p(\mu)$ and $L^p(\mu;\R^d)$ above by $L^p_\loc(\mu)$ and $L^p_\loc(\mu;\R^d)$ resp.
\end{defi}

The first two integrals in \eqref{ibpeq} are obviously well defined. The third integral is finite
by (\textbf{Diff}). It follows immediately that the gradient $\nabla^\mu$ is unique.
Also, if $f\in C^1(\R^d)$, then $f\in V^{1,p}_\loc(\mu)$ and $\nabla f=\nabla^\mu f$ $\mu$-a.s.

\begin{prop}
Assume \up{(\textbf{Diff})}. Then $V^{1,p}(\mu)$ is a Banach space with the obvious choice of a norm
\[\norm{\cdot}_{1,p,\mu}:=\left(\norm{\nabla^\mu\cdot}_{L^p(\mu;\R^d)}^p+\norm{\cdot}_{L^p(\mu)}^p\right)^{1/p}.\]

Moreover, $H^{1,p}(\mu)\subset V^{1,p}(\mu)$ and their gradients coincide $\mu$-a.e.
\end{prop}
\begin{proof}
Let $\{f_n\}\subset V^{1,p}(\mu)$ be a $\norm{\cdot}_{1,p,\mu}$-Cauchy sequence. By the Riesz-Fischer theorem,
$\{f_n\}$ converges to some $f\in L^p(\mu)$ and $\{\nabla^\mu f_n\}$ converges to some $g\in L^p(\mu;\R^d)$.
Let $i\in\{1,\ldots,d\}$ and $\eta\in C_0^\infty(\R^d)$. Passing on to
the limit in \eqref{ibpeq}
yields that
\[\int g_i\eta\phi^p\,\d x=-\int f\partial_i\eta\phi^p\,\d x-\int f\eta\beta_i\phi^p\,\d x.\]
Therefore $g=\nabla^\mu f$ and $\norm{f_n-f}_{1,p,\mu}\to 0$.

Let us prove the second part. Note that by Corollary \ref{diffcor} and the discussion above, $H^{1,p}(\mu)$
is a well defined set of elements in $L^p(\mu)$.

Let $f\in C_0^\infty(\R^d)\subset H^{1,p}(\mu)$. By (\textbf{Diff}) and the Leibniz formula for unweighted
Sobolev spaces, \eqref{Leibeq} is satisfied. By classical integration by parts, $f$ satisfies
\eqref{ibpeq} with $\nabla^\mu f=\nabla f$. We extend to all of $H^{1,p}(\mu)$ by Proposition
\ref{c0prop} using that $V^{1,p}(\mu)$ is complete.
\end{proof}

For our main result further below, we need to be able to truncate $V^{1,p}(\mu)$-functions.
Therefore, we need to verify
\emph{absolute continuity on lines parallel to the coordinate axes} in $V^{1,p}(\mu)$:

\begin{prop}\label{abscontprop}
Suppose that \up{(\textbf{Diff})} holds. Fix $1\le i\le d$.
Then $f\in V^{1,p}(\mu)$ has
a representative $\tilde{f}^i$ such that $t\mapsto\tilde{f}^i(y+te_i)$ is absolutely continuous
for (($d-1$)-dim.) Lebesgue almost all $y\in\{e_i\}^\perp$ on any compact subinterval of $R(\phi^q(y+\cdot e_i))$ if $p\in (1,\infty)$, on any compact subinterval of $\hat{R}(\phi(y+\cdot e_i))$ resp. if $p=1$.
In that case, for $\d y$-a.a. $y\in\{e_i\}^\perp$, $\d t$-a.a. $t\in R(\phi^q(y+\cdot e_i))$ (if $p\in (1,\infty)$), $\hat{R}(\phi(y+\cdot e_i))$ (if $p=1$) resp. setting $x:=y+te_i$, it holds that
\[\partial^\mu_i f(x)=\frac{\d}{\d t}\tilde{f}^i(y+te_i).\]
\end{prop}
\begin{proof}
The claim can be proved arguing similar to \cite[Proof of Lemma 2.2]{ARZ}.
Compare also with \cite[\S 4.9.2]{EG}.
\end{proof}

Picking appropriate absolutely continuous versions, one immediately obtains the following
Leibniz formula:

\begin{cor}\label{leibnizcor}
Suppose that \up{(\textbf{Diff})} holds.
If $f,g\in V^{1,p}(\mu)$ and if $fg$, $f\partial_i^\mu g$ and $g\partial_i^\mu f$ are in $L^p(\mu)$ for
all $1\le i\le d$, then $fg\in V^{1,p}(\mu)$ and $\partial_i^\mu(fg)=f\partial_i^\mu g+g\partial_i^\mu f$
for all $1\le i\le d$. Then also, $\nabla^\mu(fg)=f\nabla^\mu g+g\nabla^\mu f$.
\end{cor}

The following lemma guarantees that we can truncate Sobolev functions.

\begin{lem}\label{contractionlem}
Suppose that \up{(\textbf{Diff})} holds.
Suppose that $f\in V^{1,p}(\mu)$ and that $F:\R\to\R$ is Lipschitz. Then $F\circ f\in V^{1,p}(\mu)$ with
\[\ol{\nabla}(F\circ f)=(F'\circ f)\cdot\ol{\nabla} f\quad \mu\text{-a.s.}\]
In particular, when $F(t):=N\wedge t\vee -N$, $N\in\N$ is a cut-off function,
\begin{equation}\label{mkequation}\abs{\ol{\nabla} (F\circ f)}\le\abs{\ol{\nabla} f}\quad \mu\text{-a.s.}\end{equation}
\end{lem}
\begin{proof}
The claim can be proved arguing similar to \cite[Theorem 2.1.11]{Ziem}.
\end{proof}

\begin{lem}\label{truncationlem}
Suppose that \up{(\textbf{Diff})} holds.
The set of bounded and compactly supported functions in $V^{1,p}(\mu)$ is dense in $V^{1,p}(\mu)$.
\end{lem}
\begin{proof}
The claim follows by a truncation argument from Corollary \ref{leibnizcor} and Lemma \ref{contractionlem}.
We shall omit the proof.
\end{proof}

Note that the last two statements also hold for $H^{1,p}(\mu)$. Anyhow, the proof
of Lemma \ref{contractionlem} for $H^{1,p}(\mu)$ needs some caution, we refer to \cite[Proposition I.4.7, Example II.2.c)]{MR}.

\section{Proof of Theorem \ref{CFthm}}\label{CFsect}

We arrive at our main result. Our proof is inspired by that of Patrick Cattiaux and
Myriam Fradon in \cite{CF}. See also \cite{Fra}. However, our method in estimating \eqref{differeq} is different
from theirs, as we use maximal function-estimates instead of Fourier transforms.

For all of this section, assume \up{(\textbf{Diff})}.
By Lemma \ref{truncationlem}, bounded and compactly supported functions in $V^{1,p}(\mu)$ are dense.
We will show that a subsequence of a standard mollifier of such a function $f$ converges in $\norm{\cdot}_{1,p,\mu}$-norm to $f$.
The claim will then follow from Lemma \ref{c0prop}.

First, we need to collect some facts about the so-called \textit{centered Hardy--Littlewood maximal function} defined
for $g\in L^1_\loc(\d x)$ by
\[Mg(x):=\sup_{\rho>0}\fint_{B(x,\rho)}\abs{g(y)}\,\d y.\]
We shall need the useful inequality
\begin{equation}\label{hardyineq}\abs{u(x)-u(y)}\le c\abs{x-y}\left[M\abs{\nabla u}(x)+M\abs{\nabla u}(y)\right]\end{equation}
for any $u\in W^{1,p}(\d x)$, for all $x,y\in\R^d\setminus N$, where $N$ is a set of Lebesgue measure zero
and $c$ is a positive constant depending only on $d$ and $p$. For a proof see e.g. \cite[Corollary 4.3]{AK}.
The inequality is credited to L. I. Hedberg \cite{Hed}.

Also for all $u\in L^p(\d x)$, $p\in (1,\infty]$,
\begin{equation}\label{hardyineq2}
\norm{M u}_{L^p}\le c'\norm{u}_{L^p}\end{equation}
by the maximal function theorem \cite[Theorem I.1 (c), p. 5]{Stei1} and $c'>0$ depends only on $d$ and $p$.

For the approximation, we shall prove the following key-lemma. Compare with \cite[Lemma 2.9]{CF}.

\begin{lem}\label{CFlem}
Suppose that \up{(\textbf{Diff})} holds. Let $f\in V^{1,p}(\mu)$ such that $f$ is bounded. Then for every $\zeta\in C_0^\infty(\R^d)$
and every $1\le i\le d$
\begin{equation}\label{CFIbP}
\int\ol{\partial}_i f \zeta \phi\,\d x+\int f\partial_i \zeta \phi\,\d x+\int f\zeta\partial_i\phi\,\d x=0.
\end{equation}
In particular, $f\phi\in W^{1,1}_\loc(\d x)$ and $\partial_i(f\phi)=\phi\partial_i^\mu f+f\partial_i\phi$.
\end{lem}
\begin{proof}
For all of the proof fix $1\le i\le d$. For $p=1$, the formula follows from \eqref{ibpeq}. So, let $p\in (1,\infty)$.
Let us first assure ourselves that all three integrals in \eqref{CFIbP} are well defined.
Clearly,
\[\abs{\ol{\partial}_i f\zeta\phi}^p\le\norm{\zeta}_\infty^p\abs{\ol{\partial}_i f}^p\phi^p1_{\supp\zeta}\in L^1(\d x),\]
and hence,
\[\abs{\ol{\partial}_i f\zeta\phi}\in L^1(\d x).\]
A similar argument works for the second integral. The third integral is well defined because by $\phi\in W^{1,p}_\loc(\d x)$ we have that
\[\abs{f\zeta\partial_i\phi}^p\le\norm{f\zeta}_\infty^p\abs{\partial_i\phi}^p1_{\supp\zeta}\in L^1(\d x)\]
and hence,
\[\abs{f\zeta\partial_i\phi}\in L^1(\d x).\]
Let $M\in\N$ and $\vartheta_M\in C_0^\infty(\R)$ with
\[\vartheta_M(t)=t\;\text{for}\;t\in[-M,M],\;\abs{\vartheta_M}\le M+1,\;\abs{\vartheta_M'}\le 1\]
and
\[\supp(\vartheta_M)\subset[-3M,3M].\]
Define
\[\phi_M:=\vartheta_M\left(\frac{1}{\phi^{p-1}}\right)1_{\{\phi>0\}}.\]
Clearly, $\phi_M\in L^p_\loc$. Furthermore, define
\[\Phi_M:=(1-p)\vartheta_M'\left(\frac{1}{\phi^{p-1}}\right)\frac{\partial_i\phi}{\phi^p}1_{\{\phi>0\}}.\]
Since $\vartheta_M'(1/\phi^{p-1})\equiv 0$ on $\{\phi^{p-1}\le 1/(3M)\}$ and
\[\abs{\Phi_M}\le(p-1)\frac{\abs{\partial_i\phi}}{\phi^p}1_{\{\phi^{p-1}>1/(3M)\}}=(p-1)\frac{\abs{\partial_i\phi}}{\phi^p}1_{\{\phi^{p}>(1/(3M))^q\}},\]
hence $\Phi_M\in L^p_\loc$. We claim that $\phi_M\in W^{1,p}_\loc(\d x)$ and that $\partial_i\phi_M=\Phi_M$.
Let $\eps>0$ and define
\[\phi_M^\eps:=\vartheta_M\left(\frac{1}{(\phi+\eps)^{p-1}}\right).\]
Clearly, $\phi_M^\eps\to\phi_M$ in $L^p_\loc$ as $\eps\searrow 0$. Also, by the chain rule for Sobolev functions
(see e.g. \cite[Theorem 2.1.11]{Ziem}),
\[\partial_i\phi_M^\eps=(1-p)\vartheta_M'\left(\frac{1}{(\phi+\eps)^{p-1}}\right)\frac{\partial_i\phi}{(\phi+\eps)^p}
1_{\{\phi+\eps>(3M)^{-1/(p-1)}\}}\]
and
\[\abs{\partial_i\phi_M^\eps}\le (p-1)\frac{\abs{\partial_i\phi}}{(\phi+\eps)^p}1_{\{(\phi+\eps)^{p}>(1/(3M))^q\}}\in L^p_\loc.\]
Hence $\phi_M^\eps\in W^{1,p}_\loc(\d x)$ and $\partial_i\phi_M^\eps\to\Phi_M$ in $L^p_\loc$ as $\eps\searrow 0$.

Since $\phi\in W^{1,p}_\loc(\d x)$ and since $\phi_M$ is bounded, we have that $\phi_M\partial_i\phi\in L^p_\loc$. Also, $\phi\partial_i\phi_M\in L^p_\loc$, since
\begin{equation}\label{betaeq}\abs{\phi\partial_i\phi_M}\le(p-1)\frac{\abs{\partial_i\phi}}{\phi^{p-1}}1_{\{\phi^{p-1}>1/(3M)\}}\le (p-1) 3M\abs{\partial_i \phi}.\end{equation}
Now by the usual Leibniz rule for weak derivatives
\[\phi\phi_M\in W^{1,p}_\loc(\d x)\quad\text{and}\quad\partial_i(\phi\phi_M)=\phi_M\partial_i\phi+(1-p)\vartheta_M'\left(\frac{1}{\phi^{p-1}}\right)\frac{\partial_i\phi}{\phi^{p-1}}\]
where by definition $\partial_i\phi/\phi^{p-1}\equiv 0$ on $\{\phi=0\}$. Consider the term $\phi_M\phi^p$.
Recall that $\phi^p\in W^{1,1}_\loc(\d x)$ by Lemma \ref{pqlem}. As already seen, $\phi\phi_M\in W^{1,p}_\loc(\d x)$.
By Lemma \ref{Wqlem}, $\phi^{p-1}\in W^{1,q}_\loc(\d x)$ and $\partial_i(\phi^{p-1})=(p-1)\phi^{p-2}\partial_i\phi\in L^q_\loc$.
Hence $\phi\phi_M(\partial_i(\phi^{p-1}))\in L^1_\loc$ and $\partial_i(\phi\phi_M)\phi^{p-1}\in L^1_\loc$. It follows that $\phi_M\phi^p\in W^{1,1}_\loc(\d x)$ and by
the Leibniz rule for weak derivatives
\[\partial_i(\phi_M\phi^p)=p\phi_M\phi^{p-1}\partial_i\phi+(1-p)\vartheta_M'\left(\frac{1}{\phi^{p-1}}\right)\partial_i\phi\in L^1_\loc.\]
Let $\zeta\in C_0^\infty(\R^d)$. Applying integration by parts, we see that
\begin{equation}
\int\partial_i \zeta\phi_M\phi^p\,\d x=-p\int \zeta\phi_M\frac{\partial_i\phi}{\phi}\phi^p\,\d x+(p-1)\int \zeta\frac{\partial_i\phi}{\phi^p}\vartheta_M'\left(\frac{1}{\phi^{p-1}}\right)\phi^p\,\d x.
\end{equation}
Moreover, by \eqref{betaeq}, $\partial_i\phi_M\in L^p_\loc(\phi^p\,\d x)$. $\phi_M\in L^p_\loc(\phi^p\,\d x)$ is clear. Therefore
$\phi_M\in V^{1,p}_{\loc}(\mu)$ and
\[\ol{\partial}_i\phi_M=(1-p)\frac{\partial_i\phi}{\phi^p}\vartheta_M'\left(\frac{1}{\phi^{p-1}}\right).\]
The Leibniz rule in Corollary \ref{leibnizcor} also holds in $V^{1,p}_{\loc}(\mu)$, and so we would like to give sense to the expression
$\ol{\partial}_i(f\phi_M)=\phi_M\ol{\partial}_i f+f\ol{\partial}_i\phi_M$. But $\phi_M\in V^{1,p}_{\loc}(\mu)$, $f\in V^{1,p}(\mu)$ and $f$ is bounded,
$f\ol{\partial}_i\phi_M\in L^p_{\loc}(\mu)$ since $f$ is bounded and finally $\phi_M\ol{\partial}_i f\in L^p_{\loc}(\mu)$ since $\phi_M$ is bounded.
Hence $f\phi_M\in V^{1,p}_{\loc}(\mu)$ and the Leibniz rule holds (locally). By definition of $\ol{\partial}_i$ for $\zeta\in C_0^\infty(\R^d)$
\begin{equation}\label{Mlimeq}\begin{split}
\int\ol{\partial}_i f\zeta\phi_M\phi^p\,\d x=&(p-1)\int f\zeta\frac{\partial_i\phi}{\phi^p}\vartheta_M'\left(\frac{1}{\phi^{p-1}}\right)\phi^p\,\d x\\
&-\int f\partial_i \zeta\phi_M\phi^p\,\d x-p\int f\zeta\phi_M\frac{\partial_i\phi}{\phi}\phi^p\,\d x
\end{split}\end{equation}
Now let $M\to\infty$ in \eqref{Mlimeq}. Note that
\[\phi_M\to(1/\phi^{p-1})1_{\{\phi>0\}}\]
$\d x$-a.s. and
\[\vartheta_M'(1/\phi^{p-1})\to 1\]
$\d x$-a.s.
In order to apply Lebesgue's dominated convergence theorem, we verify
\[\abs{\ol{\partial}_i f \zeta\phi_M\phi^p}\le 2\abs{\ol{\partial}_i f\phi}\norm{\zeta}_\infty 1_{\supp \zeta}\in L^1(\d x),\]
where we have used that
\[\abs{\phi_M\phi^{p-1}}\le 1,\]
because $\vartheta_M$ is Lipschitz and $\vartheta_M(0)=0$.
Furthermore,
\begin{align*}
\abs{f\zeta\partial_i\phi\vartheta_M'\left(1/\phi^{p-1}\right)}&\le\abs{f\partial_i\phi}\norm{\zeta}_\infty 1_{\supp \zeta}\in L^1(\d x),\\
\abs{f\partial_i \zeta\phi_M\phi^p}&\le 2\abs{f\phi}\norm{\partial_i \zeta}_\infty 1_{\supp \zeta}\in L^1(\d x),\\
&\text{and}\\
\abs{f\zeta\phi_M\partial_i\phi\phi^{p-1}}&\le 2\abs{f\partial_i \phi}\norm{\zeta}_\infty 1_{\supp \zeta}\in L^1(\d x).
\end{align*}
The formula obtained, when passing on to the limit $M\to\infty$ in \eqref{Mlimeq}, is exactly the desired statement.
\end{proof}

Below, we shall need a lemma on difference quotients. Compare with \cite[Proof of Lemma 7.23]{GilTru} and \cite[Theorem 2.1.6]{Ziem}.

\begin{lem}\label{diffqlem}
Let $z\in B(0,1)\subset\R^d$ and $u\in W^{1,p}(\d x)$.
Set for $\eps>0$
\[\Delta_\eps u(x):=\frac{u(x-\eps z)-u(x)}{\eps}\]
for some representative of $u$.
Then
\[\norm{\Delta_\eps u+\lrbr{\nabla u}{z}}_{L^p(\d x)}\to 0\]
as $\eps\searrow 0$.
\end{lem}
\begin{proof}
Start with $u\in C^1\cap W^{1,p}(\d x)$. By the fundamental theorem of calculus
\[\Delta_\eps u(x)=-\frac{1}{\eps}\int_0^\eps \lrbr{\nabla u(x-sz)}{z}\,\d s.\]
Use Fubini's theorem to get
\begin{equation}\label{epsintint}\int\lrabs{\Delta_\eps u(x)+\lrbr{\nabla u(x)}{z}}^p\,\d x=\frac{1}{\eps}\int_{0}^\eps\int\lrabs{\lrbr{\nabla u(x-sz)}{z}-\lrbr{\nabla u(x)}{z}}^p\,\d x\,\d s.\end{equation}
By a well known property of $L^p$-norms \cite[p. 63]{Stei1} the map
\[s\mapsto\int\lrabs{\lrbr{\nabla u(x-sz)}{z}-\lrbr{\nabla u(x)}{z}}^p\,\d x\]
is continuous in zero. Hence $s=0$ is a Lebesgue point of this map. Therefore the right-hand side of
\eqref{epsintint} tends to zero as $\eps\searrow 0$. The claim can be extended to functions in
$W^{1,p}(\d x)$ by an approximation by smooth functions as e.g. in \cite[Theorem 2.3.2]{Ziem}.
\end{proof}

\begin{proof}[Proof of Theorem \ref{CFthm}]
Let $f\in V^{1,p}(\mu)$ be (a class of) a function which is bounded and compactly supported. By Lemma \ref{truncationlem},
we are done if we can approximate $f$ by $C_0^\infty$-functions. Let $\{\eta_\eps\}_{\eps>0}$ be a
standard mollifier. Since $f$ is bounded and compactly supported,
$\eta_\eps\ast f\in C_0^\infty(\R^d)$ with $\supp(\eta_\eps\ast f)\subset\supp f+\eps B(0,1)$ and $\abs{\eta_\eps\ast f}\le\norm{f}_\infty$.
We claim that there exists a sequence $\eps_n\searrow 0$ such that $\eta_{\eps_n}\ast f$ converges to $f$ in $V^{1,p}(\mu)$.
The $L^p(\mu)$-part is easy. Since $\eta_\eps\ast f,f\in L^1(\d x)$, $\lim_{\eps\searrow 0}\norm{\eta_\eps\ast f-f}_{L^1(\d x)}=0$. Therefore
we can extract a subsequence $\{\eps_n\}$ such that $\eta_{\eps_n}\ast f\to f$ $\d x$-a.s. For $\eps_n\le 1$
\[\abs{(\eta_{\eps_n}\ast f)\phi-f\phi}^p\le 2^p\norm{f}_\infty^p\abs{\phi}^p 1_{\supp f+B(0,1)}\in L^1(\d x).\]
By Lebesgue's dominated convergence theorem, $\lim_n\norm{\eta_{\eps_n}\ast f-f}_{L^p(\mu)}=0$.

Fix $1\le i\le d$.
We are left to prove $\partial_i(\eta_{\eps_n}\ast f)\to\ol{\partial}_i f$ in $L^p(\mu)$ for some sequence $\eps_n\searrow 0$.
Or equivalently,
\[\phi\partial_i(\eta_{\eps_n}\ast f)\to\phi\ol{\partial}_i f\quad\text{in}\;L^p(\d x).\]
Write
\begin{equation}\label{rhs}\begin{split}
&\int\abs{\phi\partial_i(\eta_\eps\ast f)-\phi\ol{\partial}_i f}^p\,\d x\\
\le&2^{p-1}\left[\int\abs{\phi\ol{\partial}_i f-(\eta_\eps\ast(\phi\ol{\partial}_i f))}^p\,\d x+\int\abs{(\eta_\eps\ast(\phi\ol{\partial}_i f))-\phi\partial_i(\eta_\eps\ast f)}^p\,\d x\right].
\end{split}\end{equation}
The first term tends to zero as $\eps\searrow 0$ by a well known fact \cite[Theorem III.2 (c), p. 62]{Stei1}. We continue with studying the second term.
Recall that $\eta_\eps(x)=\eta_\eps(\abs{x})$.
\begin{align*}
&\int\abs{\phi\partial_i(\eta_\eps\ast f)-(\eta_\eps\ast(\phi\ol{\partial}_i f))}^p\,\d x\\
=&\int\lrabs{\phi(x)\int\partial_i\eta_\eps(x-y)f(y)\,\d y-\int\eta_\eps(x-y)\phi(y)\ol{\partial}_i f(y)\,\d y}^p\,\d x\\
=&\int\bigg\lvert\int\partial_i\eta_\eps(x-y)f(y)[\phi(x)-\phi(y)]\,\d y\\
&+\int\partial_i\eta_\eps(x-y)f(y)\phi(y)-\eta_\eps(x-y)\phi(y)\ol{\partial}_i f(y)\,\d y\bigg\rvert^p\,\d x\\
&\text{apply Lemma \ref{CFlem} with $\zeta(y):=\eta_\eps(x-y)$}\\
&\text{and noting that $\partial_i\eta_\eps(x-y)=\frac{\partial}{\partial x_i}\eta_\eps(x-y)=-\frac{\partial}{\partial y_i}\eta_\eps(x-y)$}\\
=&\int\lrabs{\int\partial_i\eta_\eps(x-y)f(y)[\phi(x)-\phi(y)]\,\d y+\int\eta_\eps(x-y)f(y)\partial_i\phi(y)\,\d y}^p\,\d x\\
\le&2^{p-1}\left[\int\lrabs{\int\partial_i\eta_\eps(x-y)f(y)[\phi(x)-\phi(y)]\,\d y}^p\,\d x
+\int\lrabs{\eta_\eps\ast (f\partial_i\phi)}^p\,\d x\right]\\
\le&2^{p-1}\int\lrabs{\int\partial_i\eta_\eps(x-y)f(y)[\phi(x)-\phi(y)]\,\d y}^p\,\d x+2^{p-1}\norm{f\partial_i\phi}_{L^p(\d x)}^p.
\end{align*}
We would like to control the first term. Replace $\phi$ by $\hat{\phi}\in W^{1,p}(\d x)$ defined by:
\[\hat{\phi}=\phi\xi\;\;\text{with}\;\;\xi\in C_0^\infty(\R^d)\;\;\text{and}\;\;1_{\supp f+B(0,2)}\le\xi\le 1_{\supp f+B(0,3)}.\]
Let $h_\eps:\R^d\to\R^d$, $h_\eps(x):=-\eps x$. Then upon substituting $y=x+\eps z$ (which leads to $\d y=\eps^d\,\d z$)
\begin{align*}
&\int\lrabs{\int\partial_i\eta_\eps(x-y)f(y)\left[\phi(x)-\phi(y)\right]\,\d y}^p\,\d x\\
=&\int\lrabs{\int_{B(0,1)}\partial_i\eta_\eps(-\eps z)f(x+\eps z)\left[\hat{\phi}(x)-\hat{\phi}(x+\eps z)\right]\eps^d\,\d z}^p\,\d x
\end{align*}
By the chain rule $-\eps(\partial_i\eta_\eps)(-\eps z)=\partial_i(\eta_\eps\circ h_\eps)(z)=(1/\eps^d)\partial_i(\eta)(z)$
and hence the latter is equal to
\begin{align*}
&\int\lrabs{\int_{B(0,1)}\partial_i\eta(z)f(x+\eps z)\frac{\hat{\phi}(x)-\hat{\phi}(x+\eps z)}{\eps}\,\d z}^p\,\d x\\
\le&2^{p-1}\int\lrabs{\int_{B(0,1)}\partial_i\eta(z)f(x+\eps z)\lrbr{-\nabla\hat{\phi}(x+\eps z)}{z}\,\d z}^p\,\d x\\
&+2^{p-1}\\
&\times\int\lrabs{\int_{B(0,1)}\partial_i\eta(z)f(x+\eps z)\left[\frac{\hat{\phi}(x)-\hat{\phi}(x+\eps z)}{\eps}+\lrbr{\nabla\hat{\phi}(x+\eps z)}{z}\right]\,\d z}^p\,\d x
\end{align*}
By Jensen's inequality and Fubini's theorem, the first term is bounded by
\[C(p,d)\norm{\partial_i\eta}_\infty^p\sum_{j=1}^d\norm{f\partial_j\phi}_{L^p(\d x)}^p,\]
where $C(p,d)$ is a positive constant depending only on $p$ and $d$.

Concerning the second term, we use again Jensen's inequality and Fubini's theorem to see that it is bounded by
\begin{equation}\label{Cpdeq}C'(p,d)\norm{\partial_i\eta}_\infty^p\norm{f}_\infty^p\int_{B(0,1)}\int\lrabs{\frac{\hat{\phi}(x)-\hat{\phi}(x+\eps z)}{\eps}+\lrbr{\nabla\hat{\phi}(x+\eps z)}{z}}^p\,\d x\,\d z,\end{equation}
where $C'(p,d)$ is a positive constant depending only on $p$ and $d$. Let us investigate the inner integral.

By variable substitution, we get that the inner integral in \eqref{Cpdeq} is equal to
\begin{equation}\label{differeq}\int\lrabs{\frac{\hat{\phi}(x-\eps z)-\hat{\phi}(x)}{\eps}+\lrbr{\nabla\hat{\phi}(x)}{z}}^p\,\d x.\end{equation}
By Lemma \ref{diffqlem}, the term converges to zero pointwise as $\eps\searrow 0$ for each fixed $z\in B(0,1)$.

By inequality \eqref{hardyineq},
for $\d z$-a.a. $z\in B(0,1)$
\begin{multline*}\int\lrabs{\frac{\hat{\phi}(x-\eps z)-\hat{\phi}(x)}{\eps}+\lrbr{\nabla\hat{\phi}(x)}{z}}^p\,\d x\\
\le C(p,d)\norm{M\abs{\nabla\hat{\phi}}}_{L^p(\d x)}^p\abs{z}^p 1_{B(0,1)}(z),\end{multline*}
where $M$ denotes the centered Hardy-Littlewood maximal function.
If $p\in (1,\infty)$, then $\hat{\phi}\in W^{1,p}(\d x)$ and the right-hand side is in
$L^1(\d z)$ by estimate \eqref{hardyineq2}.
If $p=1$, then $\nabla\hat{\phi}\in L^\infty(\d x)$ by \eqref{LipGrad} and
\begin{multline*}\int\lrabs{\frac{\hat{\phi}(x-\eps z)-\hat{\phi}(x)}{\eps}+\lrbr{\nabla\hat{\phi}(x)}{z}}\,\d x\\
\le C(d,\supp f)\norm{M\abs{\nabla\hat{\phi}}}_{L^\infty(\d x)}\abs{z}^p 1_{B(0,1)}(z)\end{multline*}
and the right-hand side is again in
$L^1(\d z)$ by estimate \eqref{hardyineq2}.

The desired convergence to zero as $\eps\searrow 0$ follows now by the preceding discussion and
Lebesgue's dominated convergence theorem.

We have proved that
\begin{equation}\label{thetaestimate}
\begin{split}
&\int\abs{\phi\partial_i(\eta_\eps\ast f)-(\eta_\eps\ast(\phi\ol{\partial}_i f))}^p\,\d x\\
\le&C(d,p,\supp f,\eta)\left[\sum_{j=1}^d\norm{f\partial_j\phi}_{L^p(\d x)}^p+\norm{f}_\infty^p\theta(\eps)\right]
\end{split}\end{equation}
with $\theta(\eps)\to 0$ as $\eps\searrow 0$, and $\theta$ depends only on $\supp f$.

We shall go back to the right-hand side of \eqref{rhs}. Let $f_\delta:=\eta_\delta\ast f$ for $\delta>0$. By Lebesgue's dominated convergence theorem again, we can prove
that there is a subnet (also denoted by $\{f_\delta\}$), such that
\begin{equation}\label{deltaconvbeta}\sum_{j=1}^d\norm{(f-f_\delta)\partial_j\phi}_{L^p(\d x)}^p\to 0\end{equation}
as $\delta\searrow 0$. Taking \eqref{thetaestimate} into account, ($f$ replaced by $f-f_\delta$ therein), we get that
\begin{align*}
&\norm{\phi\partial_i(\eta_\eps\ast f)-(\eta_\eps\ast(\phi\ol{\partial}_i f))}^p_{L^p(\d x)}\\
\le&2^{p-1}\norm{\phi\partial_i(\eta_\eps\ast (f-f_\delta))-(\eta_\eps\ast(\phi\ol{\partial}_i (f-f_\delta)))}^p_{L^p(\d x)}\\
&+2^{p-1}\norm{\phi\partial_i(\eta_\eps\ast f_\delta)-(\eta_\eps\ast(\phi\ol{\partial}_i f_\delta))}^p_{L^p(\d x)}\\
\le&C(d,p,\supp f)\left[\sum_{j=1}^d\norm{(f-f_\delta)\partial_j\phi}_{L^p(\d x)}^p+\norm{f-f_\delta}_\infty^p\theta(\eps)\right]\\
&+2^{p-1}\norm{\phi\partial_i(\eta_\eps\ast f_\delta)-(\eta_\eps\ast(\phi\ol{\partial}_i f_\delta))}^p_{L^p(\d x)}.
\end{align*}
The use of \eqref{thetaestimate} is justified, since $\hat{\phi}=\phi$ on $\supp f+B(0,2)$, thus on $\supp(f-f_\delta)+B(0,1)$.
Taking \eqref{deltaconvbeta} into account, by choosing first $\delta$ and then letting $\eps\searrow 0$, the first term above can be controlled (since $\norm{f-f_\delta}_\infty\le 2\norm{f}_\infty$).
If we can prove for any $\zeta\in C_0^\infty$
\begin{equation}\label{betalast}
\norm{\phi\partial_i(\eta_\eps\ast\zeta)-(\eta_\eps\ast(\phi\ol{\partial}_i\zeta))}^p_{L^p(\d x)}\to 0
\end{equation}
as $\eps\searrow 0$, we can control the second term above and hence are done. But
\begin{align*}
&\norm{\phi\partial_i(\eta_\eps\ast\zeta)-(\eta_\eps\ast(\phi\partial_i\zeta))}^p_{L^p(\d x)}\\
\le&\int\lrabs{\int\eta_\eps(x-y)\partial_i\zeta(y)\left[\phi(x)-\phi(y)\right]\,\d y}^p\,\d x.
\end{align*}
Substituting $y=x+\eps z$ ($\d y=\eps^d\,\d z$) and using Jensen's inequality and Fubini's theorem again,
the latter is dominated by
\[
C(d,p)\norm{\eta}_\infty^p\norm{\partial_i\zeta}^p_\infty\int_{B(0,1)}\norm{(\phi\xi_\zeta)(\cdot)-(\phi\xi_\zeta)(\cdot+\eps z)}^p_{L^p(\d x)}\,\d z,
\]
where $\xi_\zeta\in C_0^\infty(\R^d)$ with $\xi_\zeta\equiv 1$ on $\supp \zeta+B(0,1)$.
\[\norm{(\phi\xi_\zeta)(\cdot)-(\phi\xi_\zeta)(\cdot+\eps z)}^p_{L^p(\d x)}\]
tends to zero as $\eps\searrow 0$ again by \cite[p. 63]{Stei1}.
By inequalities \eqref{hardyineq} and \eqref{hardyineq2} for $\d z$-a.a. $z\in B(0,1)$
\[\norm{(\phi\xi_\zeta)(\cdot)-(\phi\xi_\zeta)(\cdot+\eps z)}^p_{L^p(\d x)}\le c(d,p)\norm{\nabla(\phi\xi_\zeta)}_{L^p(\d x)}^p\abs{\eps z}^p1_{B(0,1)}\in L^1(\d z),\]
for $p\in (1,\infty)$,
and together with \eqref{LipGrad}, for $p=1$,
\begin{multline*}\norm{(\phi\xi_\zeta)(\cdot)-(\phi\xi_\zeta)(\cdot+\eps z)}_{L^1(\d x)}\\
\le c(d,p,\supp f)\norm{\nabla(\phi\xi_\zeta)}_{L^\infty(\d x)}\abs{\eps z} 1_{B(0,1)}\in L^1(\d z).
\end{multline*}
Thus we can apply Lebesgue's dominated convergence theorem.

The proof is complete.\end{proof}

\section{The Kufner-Sobolev space $W^{1,p}(\mu)$}\label{Kufsec}

We shall briefly deal with the Kufner-Sobolev space $W^{1,p}(\mu)$ first introduced in \cite{Kuf} and studied e.g. in \cite{KO,KufSae,Oh}.

\begin{defi}
Assume \up{(\textbf{Reg})}. Let
\[W^{1,p}(\mu):=\left\{u\in L^p(\mu),\;\vert\;\D u\in L^p(\mu;\R^d)\right\}.\]
\end{defi}

Note that in the above definition, by \up{(\textbf{Reg})} and Lemma \ref{locallyintegrablelem},
$u\in L^1_\loc$ and hence $\D u$ is well defined.

\begin{prop}\label{Wprop}
Assume \up{(\textbf{Reg})}. Then $W^{1,p}(\mu)$ is a Banach space with the obvious choice of a norm.
Also, by definition $H^{1,p}(\mu)\subset W^{1,p}(\mu)$. Moreover, for all $u\in H^{1,p}(\mu)$, $\nabla^\mu u=\D u$ $\d x$-a.s.
\end{prop}
\begin{proof}
See \cite[Theorem 1.11]{KO} and \cite[\S 1.9]{HKM}.
\end{proof}

The following well known result demonstrates
the power of maximal functions. We include its proof for the sake of completeness.
\begin{prop}\label{Muckprop}
Assume $1<p<\infty$. Assume that there is a global constant $K>0$ such that
\begin{equation}\label{Muckeq2}
\left(\fint_B \phi^p\,\d x\right)\cdot\left(\fint_B \phi^{-q}\,\d x\right)^{p-1}\le K,\end{equation}
for all balls $B\subset\R^d$. Then $H^{1,p}(\mu)=W^{1,p}(\mu)$.
\end{prop}
\begin{proof}
Let $f\in W^{1,p}(\mu)$, $f$ bounded and compactly supported. Let $\{\eta_\eps\}_{\eps>0}$
be a standard mollifier. $\phi$ satisfying condition \eqref{Muckeq2} is equivalent in
saying that $\phi^p=w\in A_p$, where $A_p$ is the so-called \emph{$p$-Muckenhoupt class}.
Note that this implies \up{(\textbf{Reg})}.
Let
\[Mf(x):=\sup_{\rho>0}\fint_{B(x,\rho)}\abs{f(y)}\,\d y.\]
be the centered Hardy-Littlewood maximal function of $f$. By \cite[Ch. II, \S 2, p. 57]{Stei2} we
have the pointwise estimate
\[\abs{f\ast\eta_\eps}\le Mf\quad\forall \eps>0.\]
Also, by \cite[Ch. V, \S 2, p. 198]{Stei2}, and the sublinearity of $M$ it is easy to prove that
\[\abs{\nabla(f\ast\eta_\eps)}\le M\abs{\D f}\quad\forall \eps>0.\]
By \cite[Ch. V, \S 3, p. 201, Theorem 1]{Stei2}, $w\in A_p$ implies that there exists a
constant $C>0$ such that
\[\int (Mf(x))^p\,w(x)\,\d x\le C\int\abs{f(x)}^p\,w(x)\,\d x\quad\forall f\in L^p(\mu).\]
Since $f$ was assumed bounded and compactly supported, by \up{(\textbf{Reg})}, $f\in L^1(\d x)$ and $\{f\ast\eta_\eps\}$ converges to $f$ in
$L^1(\d x)$ as $\eps\downarrow 0$. A similar statement holds for $\abs{\D f}$. Hence a subsequence converges $\d x$-a.e. Taking the above
estimates into account, we see that a subsequence of $\{f\ast\eta_\eps\}$ converges in $W^{1,p}(\mu)$ to $f$ by Lebesgue's dominated convergence theorem.
\end{proof}

We arrive at our major contribution to the study of the ``classical'' weighted Sobolev space $W^{1,p}(\mu)$.
For $p=2$ it was proved in \cite{AR3}.

\begin{prop}
Assume \up{(\textbf{Reg})}, \up{(\textbf{Diff})}, and if $p=1$, assume also that \eqref{LipGrad} holds. Then
\[H^{1,p}(\mu)=V^{1,p}(\mu)=W^{1,p}(\mu).\]
\end{prop}
\begin{proof}
The first equality follows from Theorem \ref{CFthm}. Therefore by Proposition \ref{Wprop}, $V^{1,p}(\mu)\subset W^{1,p}(\mu)$
and for $u\in V^{1,p}(\mu)$, $\nabla^\mu u=\D u$ both $\mu$-a.e. and $\d x$-a.e. (recall that \up{(\textbf{Reg})} implies
that $\d x$ and $\mu$ are equivalent measures).

Conversely, let $f\in W^{1,p}(\mu)\cap L^\infty(\mu)$.
Since by Lemma \ref{pqlem}, $\phi^p\in W^{1,1}_\loc(\d x)$, we have for each $i\in\{1,\ldots,d\}$ and each $\eta\in C_0^\infty(\R^d)$
that
\[\int\D_i f\eta\phi^p\,\d x=-\int f\partial_i(\eta\phi^p)\,\d x,\]
where $\partial_i$ is the usual weak derivative in $W^{1,1}_\loc(\d x)$. But, again by Lemma \ref{pqlem}, the right-hand side is equal to
\[-\int f\partial_i\eta\phi^p\,\d x-\int f\eta\beta_i\phi^p\,\d x.\]
Therefore $f\in V^{1,p}(\mu)$ and $\D f=\nabla^\mu f$ both $\mu$-a.e. and $\d x$-a.e.
It is well known that, given \up{(\textbf{Reg})}, bounded functions in $W^{1,p}(\mu)$ are dense in $W^{1,p}(\mu)$ and hence
$W^{1,p}(\mu)\subset V^{1,p}(\mu)$.
\end{proof}

\section{The weighted $p$-Laplacian evolution problem}\label{evo_sec}

Main result \ref{CFthm} can be used to investigate the evolution problem
related to the weighted $p$-Laplacian equation. We shall briefly illustrate the procedure
for the so-called \emph{degenerate} case, that is, $p\in [2,\infty)$.
With a \emph{weak solution} to equation (1.7), we mean a variational solution in the sense of \cite[Ch. 4.1, Theorem 4.10]{barbu2010nonlinear}.

\begin{thm}\label{exuniquethm}
Let $p\in [2,\infty)$.
Suppose also
that $\mu$ is a finite measure, so that $L^p(\mu)\subset L^2(\mu)$ densely and continuously.
Suppose that \up{(\textbf{Diff})}
holds for $\phi^p=w\ge 0$. Then the evolution problem \eqref{eeq} admits a unique (weak) solution.
\end{thm}
\begin{proof}
We represent the monotone operator
\[\begin{split}&A:V^{1,p}(\mu)\to (V^{1,p}(\mu))^\ast,\\ &\gdualdel{(V^{1,p}(\mu))^\ast}{V^{1,p}(\mu)}{A(u)}{v}=\int\abs{\nabla^\mu u}^{p-2}\lrbr{\nabla^\mu u}{\nabla^\mu v}\,w\,\d x,\end{split}\]
as the G\^ateaux derivative of
\[E_0^\mu(u):=\frac{1}{p}\int\abs{\nabla^\mu u}^p\,w\,\d x\]
in the triple of dense and continuous embeddings
$V^{1,p}(\mu)\subset L^2(\mu)\subset (V^{1,p}(\mu))^\ast$.
Since $p\ge 2$, the operator is demicontinuous, compare with \cite[Ch. 2.4, Theorem 2.5]{barbu2010nonlinear}. Boundedness of the operator $A$
follows straightforwardly. See \cite[Ch. 4.1, Theorem 4.10]{barbu2010nonlinear} for details and the terminology.
Existence follows now from \cite[Theorem 4.4]{BianWebb}. Uniqueness follows from
monotonicity.
\end{proof}

For $p\in (2,\infty)$, consider the following additional condition on $w$:
\begin{equation}\label{p2p-cond}
w^{-1/(p-2)}\in L^1(\d x)
\end{equation}

\begin{lem}\label{embedlem}
Let $p\in (2,\infty)$.
Assume that condition \eqref{p2p-cond} is satisfied for $w$.
Then $L^p(\mu)\subset L^2(\d x)$ continuously and \up{(\textbf{Reg})} is satisfied.
\end{lem}
\begin{proof}
Let $u\in L^p(\R^d,\mu)$. By H\"older's inequality,
\[\left(\int\abs{u}^2\,\d x\right)^{1/2}\le\left(\int\abs{u}^p\, \phi^p\,\d x\right)^{1/p}\cdot
\left(\int\left(\frac{1}{\phi}\right)^{p/(p-2)}\,\d x\right)^{(p-2)/(2p)},\]
which is finite by \eqref{p2p-cond}.

Since for any ball $B\subset\R^d$, it holds that
\[1_B w^{-1/(p-1)}\le \left( 1_B w^{-1/(p-2)}+1_B\right)\in L^1_\loc(\d x),\]
we see that \up{(\textbf{Reg})} is satisfied.
\end{proof}

Consider the following evolution equation in $L^2(\d x)$
\begin{equation}\label{eeq2}\left.\begin{aligned}\partial_t u&=\div\left[w\abs{\nabla u}^{p-2}\nabla u\right],&&\quad\text{in}\;\;(0,T)\times\R^d,\\
u(\cdot,0)&=u_0\in L^2(\d x),&&\quad\text{in}\;\;\R^d.\end{aligned}\right\}\end{equation}

The above equation differs in a ``weight term'' due to the dualization in $L^2(\d x)$ rather
than in $L^2(\mu)$.

\begin{thm}
Let $p\in (2,\infty)$.
Assume that condition \eqref{p2p-cond} is satisfied for $w$. Assume \up{(\textbf{Diff})}.
Then
\eqref{eeq2} admits a unique (weak) solution.
\end{thm}
\begin{proof}
By Lemma \ref{embedlem}, $L^p(\mu)\subset L^2(\d x)$ continuously. Hence the proof follows again from \cite[Theorem 4.4]{BianWebb} and
monotonicity.
\end{proof}

\section{A new class of $p$-admissible weights}

We shall recall the definition of $p$-admissible weights from \cite{HKM} by Heinonen, Kilpel\"ailen and Martio.
Note the similarities between \eqref{HKMclos} and \eqref{clos} above.

\begin{defi}\label{paddefi}
A weight $w\in L^1_\loc(\R^d)$, $w\ge 0$ is called \textup{$p$-admissible} if the following four conditions are satisfied.
\begin{itemize}
 \item $0<w<\infty$ $\d x$-a.e. and the weight is \textup{doubling}, i.e. there is a constant $C_1>0$ such that
\begin{equation}\label{doublingeq}
 \int_{2B}w\,\d x\le C_1\int_B w\,\d x\quad\forall\text{\;balls\;}B\subset\R^d.
\end{equation}
 \item If $\Omega\subset\R^d$ is open and $\{\eta_k\}\subset C^\infty(\Omega)$ is a sequence of functions such that
\begin{equation}\label{HKMclos}\int_\Omega\abs{\eta_k}^p w\,\d x\to 0\quad\text{and}\quad\int_\Omega\abs{\nabla\eta_k-v}^p w\,\d x\to 0\end{equation}
for some $v\in L^p(\Omega,w\,\d x;\R^d)$, then $v\equiv 0\in\R^d$.
 \item There are constants $\kappa>1$ and $C_3>0$ such that
\begin{equation}\label{weightedsobolevineq}
\left(\frac{1}{\int_B w\,\d x}\int_B\abs{\eta}^{\kappa p}w\,\d x\right)^{1/(\kappa p)}\le C_3\diam B\left(\frac{1}{\int_B w\,\d x}
 \int_B\abs{\nabla\eta}^p w\,\d x\right)^{1/p},
\end{equation}
whenever $B\subset\R^d$ is a ball and $\eta\in C_0^\infty(B)$.
 \item There is a constant $C_{4}>0$ such that
\begin{equation}\label{weightedpoincare}
\int_B\abs{\eta-\eta_B}^p w\,\d x\le C_4(\diam B)^p\int_B\abs{\nabla\eta}^p w\,\d x,
\end{equation}
whenever $B\subset\R^d$ is a ball and $\eta\in C^\infty_b(B)$. Here
\[\eta_B:=\frac{1}{\int_B w\,\d x}\int_B\eta\, w\,\d x.\]
\end{itemize}
\end{defi}

The next results were basically proved by Hebisch and Zegarli\'nski in \cite[Section 2]{HeZe}.
We include the proofs in order to make this paper self-contained and obtain
concrete bounds due to a more specific situation.

\begin{lem}\label{xqlem}
Let $1<q<\infty$, $\beta\in(0,\infty)$. Let $\mu(\d x):=\exp(-\beta\abs{x}^q)\,\d x$.
Then for any $C\ge (\beta q)^{-1}$, any $\eps>0$ and any $D\ge (1+\eps)^{q-1}+(\eps^{-1}+d-1)C$, we have that
\begin{equation}\label{xqlemeq0}\int\abs{f}\abs{x}^{q-1}\,\mu(\d x)\le C\int\abs{\nabla f}\,\mu(\d x)+D\int\abs{f}\,\mu(\d x),\end{equation}
for all $f\in C^1_0(\R^d)$.
\end{lem}
\begin{proof}
Let $f\in C_0^1(\R^d)$ such that $f\ge 0$ and $f$ is equal to zero on the unit ball.
By the Leibniz rule we get that
\[(\nabla f)e^{-\beta\abs{\cdot}^q}=\nabla\left(fe^{-\beta\abs{\cdot}^q}\right)+\beta q f\abs{\cdot}^{q-1}\sign(\cdot)e^{-\beta\abs{\cdot}^q}.\]
Plugging into the functional $g\mapsto\int\lrbr{g(x)}{\sign(x)}\,\d x$ yields
\begin{equation}\label{xqlemeq1}\begin{split}&\int\lrbr{\sign(x)}{\nabla f(x)}e^{-\beta\abs{x}^q}\,\d x\\
=&\int\lrbr{\sign(x)}{\nabla\left(fe^{-\beta\abs{x}^q}\right)}\,\d x+\beta q\int f(x)\abs{x}^{q-1}e^{-\beta\abs{x}^q}\,\d x.
\end{split}\end{equation}
Clearly, for the left-hand side,
\begin{equation}\label{xqlemeq2}\int\lrbr{\sign(x)}{\nabla f(x)}e^{-\beta\abs{x}^q}\,\d x\le\int\abs{\nabla f(x)}e^{-\beta\abs{x}^q}\,\d x.\end{equation}
Recall that
\begin{equation}\label{distrdiv}\div(\sign(x))=\left\{\begin{aligned}2\delta_0,&\;\;\text{if}\;\;d=1,\\
\frac{d-1}{\abs{x}},&\;\;\text{if}\;\;d\ge 2\end{aligned}\right.\end{equation}
(in the sense of distributions), where $\delta_0$ denotes the Dirac measure in $0$.
Hence after an approximation by mollifiers, for $d=1$, we get the formula
\begin{equation}\label{xqlemeq3}\int\lrbr{\sign(x)}{\nabla\left(fe^{-\beta\abs{x}^q}\right)}\,\d x=-2\int fe^{-\beta\abs{x}^q}\,\delta_0(\d x)=-2f(0)=0.\end{equation}
For $d\ge 2$, we get that
\begin{equation}\label{distrdiv2}\begin{split}&\int\lrbr{\sign(x)}{\nabla\left(fe^{-\beta\abs{x}^q}\right)}\,\d x\\
=&(1-d)\int\frac{1}{\abs{x}} fe^{-\beta\abs{x}^q}\,\d x\ge (1-d)\int fe^{-\beta\abs{x}^q}\,\d x.\end{split}\end{equation}
Gathering \eqref{xqlemeq1}, \eqref{xqlemeq2}, \eqref{xqlemeq3} and \eqref{distrdiv2} gives
\begin{equation}\label{xqlemeq4}\beta q\int f\abs{x}^{q-1}\,\mu(\d x)\le\int\abs{\nabla f}\,\mu(\d x)+(d-1)\int f\,\mu(\d x).\end{equation}
Replacing $f$ by $\abs{f}$ and noting that $\nabla(\abs{f})=\sign(f)\nabla f$, we can extend to arbitrary $f\in C_0^1$
such that $f\equiv 0$ on $B(0,1)$.

Now, let $f\in C_0^1$ be arbitrary.
Let $\eps>0$. Let $\phi(x):=1\wedge(\eps^{-1}((1+\eps)-\abs{x})\vee 0)$.
Then $f=g+h$, where $g:=\phi f$ and $h:=(1-\phi)f$. Also, $h\equiv 0$ on $B(0,1)$. Now,
\begin{equation}\label{xqlemeq5}\begin{split}
\int\abs{f}\abs{x}^{q-1}\,\mu(\d x)&=\int_{\abs{x}\le 1+\eps}\abs{f}\abs{x}^{q-1}\,\mu(\d x)+\int_{\abs{x}>1+\eps}\abs{f}\abs{x}^{q-1}\,\mu(\d x)\\
&\le (1+\eps)^{q-1}\int_{\abs{x}\le 1+\eps}\abs{f}\,\mu(\d x)+\int_{\abs{x}>1+\eps}\abs{h}\abs{x}^{q-1}\,\mu(\d x)\\
&\le (1+\eps)^{q-1}\int\abs{f}\,\mu(\d x)+\int\abs{h}\abs{x}^{q-1}\,\mu(\d x).
\end{split}\end{equation}
Note that $\abs{\nabla h}\le\abs{\nabla f}+\eps^{-1}\abs{f}$ $\d x$-a.s. Let $C\ge(\beta q)^{-1}$.
By an approximation in $W^{1,\infty}$-norm, we see that \eqref{xqlemeq4} is also valid for $h$ and hence
\[\begin{split}
&\int\abs{h}\abs{x}^{q-1}\,\mu(\d x)\le C\int\abs{\nabla h}\,\mu(\d x)+C(d-1)\int\abs{h}\,\mu(\d x)\\
\le& C\int\abs{\nabla f}\,\mu(\d x)+(\eps^{-1}+d-1)C\int\abs{f}\,\mu(\d x),\end{split}\]
which, combined with \eqref{xqlemeq5}, yields inequality \eqref{xqlemeq0} with
$D\ge(1+\eps)^{q-1}+(\eps^{-1}+d-1)C$.
\end{proof}

\begin{lem}\label{potlem}
Let $1<p<\infty$, $q:=p/(p-1)$, $\beta\in(0,\infty)$. Let $\mu(\d x):=\exp(-\beta\abs{x}^q)\,\d x$. Let $C\ge(\beta q)^{-1}$.
Let $W\in C^1(\R^d)$ be a differentiable potential (in particular, is bounded below) such that
\begin{equation}\label{potbound}\abs{\nabla W(x)}\le\delta\abs{x}^{q-1}+\gamma\end{equation}
with some constants $0<\delta<C^{-1}$, $\gamma\in(0,\infty)$.
Let $V$ be measurable such that $\osc V:=\sup V-\inf V<\infty$. Let $\d\nu:=\exp(-W-V)\,\d\mu$.
Then for any $\eps_0>0$, any
\[C'\ge(1-C\delta)^{-1}\eps_0 p C e^{2\osc V},\]
any $\eps_1>0$ and any
\[D'\ge (1-C\delta)^{-1}e^{2\osc V}\left((1+\eps_1)^{q-1}+(\eps_1^{-1}+d-1)C+(\eps_0p)^{-q/p}Cpq^{-1}+\gamma\right)\]
it holds that
\begin{equation}\label{CDeq}
\int\abs{f}^p\abs{x}^{q-1}\,\nu(\d x)\le C'\int\abs{\nabla f}^p\,\nu(\d x)+D'\int\abs{f}^p\,\nu(\d x),
\end{equation}
for any $f\in C_0^1$.
\end{lem}
\begin{proof}
Plug $\abs{f}^p e^{-W}$ into \eqref{xqlemeq0}. By Leibniz's rule we get that
\begin{align*}
&\int\abs{f}^p\abs{x}^{q-1}e^{-W}\,\mu(\d x)\\
\le &Cp\int\abs{f}^{p-1}\abs{\nabla f}e^{-W}\,\mu(\d x)\\
&+C\int\abs{f}^p\abs{\nabla W}e^{-W}\mu(\d x)+D\int\abs{f}^p e^{-W}\mu(\d x).
\end{align*}
For the first term,
\begin{align*}
&Cp\int\abs{f}^{p-1}\abs{\nabla f}e^{-W}\,\mu(\d x)\\
\le&Cp\left(\int\abs{\nabla f}^p e^{-W}\,\mu(\d x)\right)^{1/p}\cdot\left(\int\abs{f}^p e^{-W}\,\mu(\d x)\right)^{1/q}\\
\le&\eps_0 p C\int\abs{\nabla f}^p e^{-W}\,\mu(\d x)+(\eps_0p)^{-q/p}Cpq^{-1}\int\abs{f}^p e^{-W}\,\mu(\d x),
\end{align*}
by the H\"older and Young inequalities resp. Since $\osc V<\infty$, the claim follows by an easy perturbation argument,
see e.g. \cite[preuve du th\'eor\`eme 3.4.1]{Gen}.
\end{proof}

Usually, one would set $\eps_0:=p^{-1}$ and $\eps_1:=1$.

\begin{thm}\label{pointhm}
Let $1<p<\infty$ and let $w$ be a weight such that $w$ satisfies a local $p$-Poincar\'e inequality \eqref{weightedpoincare}
with constant $C_4>0$. Let $\beta$, $W$, $V$, $C'>0$, $D'>0$ be as in Lemma \ref{potlem}.

Let $L>D'$. Let
\[a_L:=\osc\displaylimits_{B(0,L^{p-1})}\left[-\beta\abs{\cdot}^q-W-V\right].\]
Let
\[c\ge2^{q}\frac{e^{2a_L}C_4L^{p(p-1)}+\frac{C'}{L}}{1-\frac{D'}{L}}.\]
Suppose that
$\d\nu_w:=\exp(-\beta\abs{\cdot}^q-W-V)\,w\,\d x$ is a finite measure.
Then $\nu_w$ satisfies the Poincar\'e inequality
\[\int\lrabs{f-\frac{\int f\,\d\nu_w}{\int\,\d\nu_w}}^p\,\d\nu_w\le c\int\abs{\nabla f}^p\,\d\nu_w,\]
for all $f\in C_b^\infty(\R^d)$.
\end{thm}
\begin{proof}
By the results of Lemma \ref{potlem}, we can apply \cite[Theorem 3.1]{HeZe}.
\end{proof}

Before we prove Theorem \ref{p-ad-thm}, let us note that, under our assumptions,
the results of Hebisch and Zegarli\'nski (in this particular case) extend to $V^{1,p}(\mu)=W^{1,p}(\mu)$.
Of course, other Poincar\'e and Sobolev type inequalities for smooth functions extend similarly to $V^{1,p}(\mu)$ if
the weight satisfies (\textbf{Diff}).

\begin{proof}[Proof of Theorem \ref{p-ad-thm}]
Let us prove that $\exp(-\beta\abs{\cdot}^q-W-V)$ is doubling.
Let $c_{1}^{W},c_{1}^{V}\ge 1$, $c_2^W,c_2^V\in\R$ be the constants from property (D).
Let $a:=\inf W$, $b:=\inf V$. Let $B\subset\R^d$ be any ball. Then
\begin{multline*}
\int_{2B}e^{-\beta\abs{x}^q-W(x)-V(x)}\,\d x=2\int_B e^{-2^q\beta\abs{x}^q-W(2x)-V(2x)}\,\d x\\
\le 2e^{-(c^W_1-1)a+c_2^W-(c^V_1-1)b+c_2^V}\int_B e^{-\beta\abs{x}^q-W(x)-V(x)}\,\d x,
\end{multline*}
which proves the doubling property.

By similar arguments as in the proof of Lemma \ref{closlem}, condition \eqref{HKMclos} is implied condition \up{(\textbf{Reg})} which is obviously satisfied,
since $\beta\abs{\cdot}^q$, $W$ and $V$ are locally bounded.
However, by a general result due to Semmes, \eqref{HKMclos} is implied by \eqref{doublingeq} and \eqref{weightedpoincare},
see \cite[Lemma 5.6]{HeKo}.

The weighted Poincar\'e inequality \eqref{weightedpoincare} follows from Theorem \ref{pointhm}
by noting that $\exp(-\beta\abs{x}^q-W-V)\,\d x$ is a finite measure.

The weighted Sobolev inequality \eqref{weightedsobolevineq} follows from \eqref{doublingeq} and \eqref{weightedpoincare} by a general result of Haj\l{}asz and Koskela \cite{HK}.

Suppose now that $V\in W^{1,\infty}_\loc(\d x)$. Since $W\in C^1$, also $W\in W^{1,\infty}_\loc(\d x)$. A similar statement
holds for $-\beta\abs{\cdot}^q$. Therefore, it is an easy exercise to check that the conditions \up{(\textbf{Reg})} and \up{(\textbf{Diff})} are satisfied.
\end{proof}

\section*{Acknowledgements}

The author would like to thank Michael R\"ockner for his interest in the subject and several helpful discussions. The author would like
to thank Oleksandr Kutovyi for checking the proof of the main result. The author would like to express his gratitude to the referees, who have provided valuable remarks.

The author acknowledges that some results of this
work can be found in Chapter 2.6 of the book \cite{Bogachev10} by Vladimir I. Bogachev. The research on this work, however,
has been carried out independently and a preliminary form of
the results has been published in \cite{Toe3}, at about the same time when the book by Bogachev 
appeared.

\end{document}